\documentclass{article}
\usepackage{}
\usepackage{amssymb,amsmath}
\usepackage{latexsym}
\usepackage[dvips]{graphicx}
\usepackage{graphicx,color}
\usepackage{hyperref}
\usepackage{mathrsfs}

\let\cal=\mathcal
\newtheorem{theo}{Theorem}[section]

\newtheorem{problem}{Problem}
\newtheorem{lemma}[theo]{Lemma}

\newtheorem{coro}[theo]{Corollary}
\newtheorem{con}[theo]{Conjecture}

\newcommand*{\qed}{\hfill\ensuremath{\square}}

\begin{document}
\title{The Erd\H{o}s--Ko--Rado Theorem in $\ell_2$-Norm\thanks{{E-mail:  $^1$wu@hunnu.edu.cn (B. Wu), $^2$huajunzhang@usx.edu.cn (H. Zhang, Corresponding author)}}}
\date{}
\author{Biao Wu$^1$, Huajun Zhang$^2$\\[10pt]
$^{1}$MOE-LCSM, School of Mathematics and Statistics,
Hunan Normal University\\
Changsha 410081, P. R. China\\[6pt]
$^{2}$Department of Mathematics,
Shaoxing University\\
Shaoxing 312000, P. R. China\\[6pt]
}

\maketitle

\begin{abstract}
The codegree squared sum ${\rm co}_2(\cal F)$ of a family (hypergraph) $\cal F \subseteq \binom{[n]} k$ is defined to be the sum of codegrees squared $d(E)^2$ over all $E\in \binom{[n]}{k-1}$, where $d(E)=|\{F\in \cal F: E\subseteq F\}|$.
Given a family of $k$-uniform hypergraphs $\mathscr H$, Balogh, Clemen and Lidick\'y recently introduced the problem to determine the maximum codegree squared sum ${\rm co}_2(\cal F)$ over all $\mathscr H$-free $\cal F$.
In the present paper, we consider the families which have as forbidden configurations all pairs of sets with intersection sizes less than $t$,
 that is, the well-known $t$-intersecting families. We prove the following Erd\H{o}s--Ko--Rado Theorem in $\ell_2$-norm, which confirms a conjecture of Brooks and Linz.

Let $t,k,n$ be positive integers such that $t\leq k\leq n$. If a family $\mathcal F\subseteq \binom{[n]}{k}$ is $t$-intersecting, then for $n\ge (t+1)(k-t+1)$, we have
\[{\rm co}_2(\cal F)\le {\binom{n-t}{k-t}}(t+(n-k+1)(k-t)),\]
equality holds if and only if $\mathcal{F}=\{F\in {\binom{[n]}{k}}: T\subseteq F\}$ for some $t$-subset $T$ of $[n]$.

In addition, we prove a Frankl--Hilton--Milner Theorem  in $\ell_2$-norm for $t\ge 2$, and a generalized Tur\'an result, i.e., we determine the maximum number of copies of tight path of length 2 in $t$-intersecting families.
\end{abstract}


\section{Introduction}

Determining the Tur\'an numbers of $k$-uniform hypergraphs is a central problem in extremal combinatorics. This question has been extensively studied for non-bipartite graphs, but it remains notoriously difficult for hypergraphs. To gain a deeper understanding of Tur\'an problems, numerous extensions have been proposed.
Recently, Balogh, Clemen, and Lidicky introduced a new type of extremality for hypergraphs. We now introduce this concept.

For integers $m$ and $n$ ($m\le n$), denote $[m,n]=\{m,m+1,\cdots,n\}$ and $[n]=[1,n]$.
Let $\mathcal{F}$ be a $k$-uniform hypergraph on $[n]$. For a set $E\subseteq[n]$, the codegree of $E$, denoted by $d(E)$, is the number of edges in $\mathcal{F}$ that contain $E$. The codegree vector is the vector $\textbf{x}\in\mathbb{Z}^{\binom{[n]}{k-1}}$ whose entries are given by $\mathbf{x}_E=d(E)$ for all $(k-1)$-subsets $E\subseteq[n]$. The codegree squared sum ${\rm co}_2(\mathcal{F})$ is defined as the square of the $\ell_2$-norm of the codegree vector of $\mathcal{F}$, i.e.,
\[
{\rm co}_2(\mathcal{F})=\sum_{E\in\binom{[n]}{k-1}}d(E)^2.
\]

For a $k$-uniform hypergraph $\mathcal{F}$, observe that
$
\sum_{E\in\binom{[n]}{k-1}}d(E)=k|\mathcal{F}|,
$
which implies that the $\ell_1$-norm of the codegree vector is directly proportional to the classical extremal number, with the proportionality constant being $k$.
The problem of maximizing the codegree squared sum over all $k$-uniform hypergraphs with a fixed number of edges has been explored by Alon \cite{Alon} for $k=2$ and by Gruslys, Letzter, and Morrison \cite{GLM} for $k\geq 3$. Given a family of $k$-uniform hypergraphs $\mathscr{H}$, a $k$-uniform hypergraph $\mathcal{F}$ is said to be $\mathscr{H}$-free if it does not contain any copy of a member of $\mathscr{H}$. Recently, Balogh, Clemen, and Lidicky \cite{BCL,BCL2} introduced the following problem.

\begin{problem}[Balogh--Clemen--Lidick\'y, \cite{BCL}]
Given a family of $k$-uniform hypergraphs $\mathscr H$, what is the maximum $\ell_2$-norm
of the codegree vector of a $k$-uniform $\mathscr H$-free $n$-vertex hypergraph $\cal F$?
 \end{problem}

Balogh, Clemen and Lidick\'y \cite{BCL,BCL2} solved Problem 1 asymptotically for a number of 3-uniform hypergraphs; the exact results were obtained in several subsequent works.
One particularly important case involves the hypergraph $\mathscr H=\{\{F,G\}:F,G\in \binom{[n]}k, |F\cap G|\in [0,t-1]\}$ for some given integer $t$.
 A family $\mathcal{F}$  is called  {\it $t$-intersecting } if $|A\cap B|\ge t$ for all $A,B\in \mathcal{F}$.
Consequently, $\cal F$ is $\mathscr H$-free if and only if $\cal F$ is $t$-intersecting.
 The Erd\H{o}s-Ko-Rado Theorem stands as one of the cornerstone results in extremal combinatorics.

\begin{theo}[Erd\H{o}s--Ko--Rado, \cite{EKR1961}]\label{EKR1961}
Let $t,k,n$ be positive integers such that $t \leq k\leq n$. If $ \mathcal F \subseteq \binom{[n]}{k} $ is $t$-intersecting,  then for $n\geq n_0(k,t)$, \[|\mathcal{F}|\leq \binom{n-t}{k-t}. \]
\end{theo}

Let $k,t$ be fixed positive integers such that $t\le k$, and let $N_0(k,t)$ be the smallest possible value of $n_0(k,t)$ in Theorem \ref{EKR1961}. In the case $t=1$, it is shown in \cite{EKR1961} that $N_0(k,1)=2k$. For $t\ge 1$, we have $N_0(k,t)=(t+1)(k-t+1)$, which was established in \cite{Fra1976} for $t\ge 15$ and in \cite{Wil1984} for all $t$. The Erd\H{o}s--Ko--Rado Theorem  has numerous versions and represents a significant research direction in extremal combinatorics. Therefore, it is intriguing to explore  an Erd\H{o}s--Ko--Rado--type theorem in the context of the $\ell_2$-norm. To achieve this, one promising approach involves leveraging an important inequality on the codegree squared sum of hypergraphs, established by Bey \cite{Bey}. This inequality generalizes a result for graphs originally due to de Caen \cite{Caen}.

\begin{theo}[Bey, \cite{Bey}]\label{Bey}
Let $\ell,k,n$ be integers with $0\le \ell \le k \le n$, $\cal F \subseteq \binom{[n]}k$ be a $k$-uniform hypergraph. Then,
\[
\sum_{E\in \binom{[n]}{\ell}}d(E)^2 \le \frac{\binom {k} {\ell}\binom {k-1} {\ell}}{\binom {n-1} {\ell}}|\cal F|^2+\binom {k-1} {\ell-1}\binom {n-\ell-1} {k-\ell}|\cal F|.
\]
\end{theo}

Bey also noted \cite{Bey2} the relation between Theorem \ref{Bey} and the well-known Kleitman-West problem, which asks for the maximum number of pairs of edges intersecting in exactly $k-1$ elements within $\cal F\subseteq \binom{[n]}k$, over all such families with a fixed size  $|\cal F|$.
By combining Bey's inequality with the  Erd\H{o}s--Ko--Rado Theorem (Theorem \ref{EKR1961}), Brooks and Linz \cite{BL}
established an Erd\H{o}s--Ko--Rado--type theorem in the context of codegree squared extremal number for $t=1$.
In the same literature, Brooks and Linz also proved versions of the Erd\H{o}s Matching Conjecture and the $t$-intersecting Erd\H{o}s--Ko--Rado theorem for the codegree squared extremal number for sufficiently  large $n$. They determined the exact codegree squared extremal number for linear $3$-paths and $3$-cycles, and asymptotically determined the codegree squared extremal number for linear $s$-paths and $s$-cycles for $s \ge 4$.

In the present paper,  by adapting the generating set method \cite{AK1996,AK} to $\ell_2$-norm problem for $t$-intersecting families, we prove an Erd\H{o}s--Ko--Rado Theorem in $\ell_2$-norm for all $n\ge (t+1)(k-t+1)$.
This confirms a conjecture of Brooks and Linz \cite{BL}.

\begin{theo}\label{maintheo1}
Let $t,k,n$ be positive integers such that $t\leq k\leq n$. If $ \mathcal F \subseteq \binom{[n]}{k} $ is $t$-intersecting, then for $n\ge (t+1)(k-t+1)$, we have
\[{\rm co}_2(\cal F)\le {\binom{n-t}{k-t}}(t+(n-k+1)(k-t)),\]
equality holds if and only if $\mathcal{F}=\{F\in {\binom{[n]}{k}}: T\subseteq F\}$ for some $t$-subset $T$ of $[n]$.
\end{theo}

Let $\cal F$ be a $k$-graph and $\mathscr H$ be a family of $k$-graphs. The generalized Tur\'an number $ex(n,\cal F,\mathscr H)$ is defined as the maximum number of copies of $\cal F$ in $n$-vertex $k$-graph $\cal G$, subject to the constraint that $\cal G$ is $\mathscr H$-free.
When $\cal F$ is a single edge, $ex(n,\cal F,\mathscr H)=ex(n,\mathscr H)$ corresponds to the classic Tur\'an number.
Determining $ex(n,\cal F,\mathscr H)$ for some  specific   $\cal F$ and $\mathscr H$ is a challenging and important problem in extremal combinatorics.
Let $\cal P_2=\{[k],[k-1]\cup \{k+1\}\}$ be a $k$-uniform tight path of length 2, and $\mathscr P=\{\{[k], [1,i]\cup [k+1,2k-i]\}: i\in \{0,1,\cdots,t-1\}\}$ be the families of all $k$-uniform path of length $2$ with  intersection size less than $t$.
The Kleitman-West problem is a discrete edge isoperimetric problem  stated as follows: Given positive integers $n > k > 0$ and $0\le m\le \binom{n}{k}$, which hypergraph $\cal F \subseteq \binom{[n]}{k}$ with $|\cal F| = m$ maximizes the number of copies $\cal P_2$?
This problem remains open and the progress made so far can be found in the references  \cite{AKa,AC,GLM,DGS2016,Harper,Har04}.
In the present paper, we consider similar problems for forbidding $\mathscr P$ rather than for fixed $m$.
We determine $ex(n,\cal P_2,\mathscr P)$ for all $n\ge (t+1)(k-t+1)$.
\begin{theo}\label{maintheo2}
Let $t,k,n$ be positive integers such that $t\leq k\leq n$. If $n\ge (t+1)(k-t+1)$, then
\[ex(n,\cal P_2,\mathscr P) =  \frac 12(k-t)(n-k)\binom{n-t}{k-t},\]
where the unique extremal family is $\mathcal{F}=\{F\in {\binom{[n]}{k}}: T\subseteq F\}$ for some $t$-subset $T$ of $[n]$.
\end{theo}
The bound for $n$ is tight. Specifically, a $k$-graph $\cal F$ is $\mathscr P$-free if and only if $\cal F$ is $t$-intersecting. Consequently, Theorem \ref{maintheo2} also implies an Erd\H{o}s--Ko--Rado--type result: For  $n\geq(t+1)(k-t+1)$, the maximum number of copies of $\cal P_2$ in $t$-intersecting families is attained in the full $t$-star.

{\bf Remark.} 
Suppose that $\cal F \subseteq \binom{[n]}{k}$ is a $t$-intersecting family with maximality of number of copies of $\cal P_2$.
As new sets are added to  $\cal F$, the number of copies of $\cal P_2$ of the resulting family will not decrease. 
Hence, we can assume that $\cal F$ is maximal.
As Theorem \ref{maintheo2} follows by the proof of Theorem \ref{maintheo1}, do not include the proof regarding the number of edges.

We say that a $t$-intersecting family $\cal F$ is trivial if all its members share $t$ common elements, and non-trivial otherwise.
Erd\H{o}s, Ko and Rado \cite{EKR1961}  posed the problem of determining the maximum size of a non-trivial intersecting family of $k$-element subsets of $[n]$.
This question was subsequently addressed by Hilton and Milner \cite{HM} in 1967, who established the following result.

\begin{theo}[Hilton--Milner, \cite{HM}] \label{thmhm}
Let $n$ and $k$ be two positive integers with $k\ge 2$ and $n\ge 2k$. If $\mathcal{F}$ is a non-trivial intersecting family of $\binom{[n]}{k}$, then
\begin{equation*}
 |\mathcal{F}| \le {n-1 \choose k-1} -{n-k-1 \choose k-1} +1.
 \end{equation*}
 For $n> 2k$,  equality holds if and only if $\cal F\cong \cal H(n, k, 1)$, where $\cal H(n, k, 1)$ is defined in the following.
\end{theo}

Brooks and Linz \cite{BL} conjectured an analogous result of Hilton--Milner Theorem.
 It is of particular interest to explore the analogous problem in the  $\ell_2$-norm for non-trivial  $t$-intersecting families.
We
introduce two particular non-trivial $t$-intersecting families, $\cal A(n, k, t)$ and $\cal H(n, k, t)$, defined as follows:
\[\cal H(n, k, t):=\{F\subseteq \binom{[n]}{k}: [t]\subseteq F, F\cap [t+1,k+1]\neq \emptyset\} \cup \{[k+1]\setminus \{i\}:i\in [t]\},\]
\[\cal A(n, k, t):=\{F\subseteq \binom{[n]}{k}: |F\cap [t+2]|\ge t+1\}.\]

\begin{theo}[Frankl\cite{Fra1978}, Frankl and F\"{u}redi\cite{FF}, and Ahlswede and Khachatrian\cite{AK1996}]
Let $t,k,n$ be positive integers such that $1\le t\leq k\leq n$. If $ \mathcal F \subseteq \binom{[n]}{k} $ is non-trivial $t$-intersecting, then for $n\ge (t+1)(k-t+1)$, we have
\[|\cal F|\le \max\{|\cal H(n, k, t)|,|\cal A(n, k, t)|\}.\]
For $n> (t+1)(k-t+1)$,  equality holds if and only if $\cal F\cong \cal H(n, k, t)$ or $\cal F\cong \cal A(n, k, t)$.
\end{theo}
For $t = 2$ and $n > n_0(k,t)$, the result was established in \cite{Fra1978}. For $t \ge 20$ and all $n > (k-t + 1)(t + 1)$
 the conclusion can be inferred from \cite{FF}.  However, the result was first proven in full generality in \cite{AK1996}.

It is natural to study the analogous problem  in $\ell_2$-norm for $t$-intersecting families.
In the present paper, we  establish the following two Frankl--Hilton--Milner--type  Theorems for $t\ge 2$.
\begin{theo}\label{maintheo3}
Let $t,k,n$ be positive integers such that $2\le t\leq k\leq n$. If $ \mathcal F \subseteq \binom{[n]}{k} $ is non-trivial $t$-intersecting, then for $n\ge (t+1)(k-t+1)$, we have
\[{\rm co}_2(\cal F)\le \max\{{\rm co}_2(\cal H(n, k, t)),{\rm co}_2(\cal A(n, k, t))\},\]
 equality holds if and only if $\cal F\cong \cal H(n, k, t)$ or $\cal F\cong \cal A(n, k, t)$.
\end{theo}

For $\cal F\subseteq \binom{[n]}{k}$, let $\zeta_{k-1}(\cal F)$ denote the number of copies of $\cal P_2$ in $\cal F$, that is,
\[\zeta_{k-1}(\cal F)=|\{\{F,G\}:F,G \in \cal F, |F\cap G|=k-1\}|.\]
 By Theorem \ref{maintheo2}, we have $\zeta_{k-1}(\cal F)\le ex(n,\cal P_2,\mathscr P)$,  with equality holding if and only if $\cal F$ is a full $t$-star.
We obtain the second extremal size of $\zeta_{k-1}(\cal F)$ as follows.

\begin{theo}\label{maintheo4}
Let $t,k,n$ be positive integers such that $2\le t\leq k\leq n$. If $ \mathcal F \subseteq \binom{[n]}{k} $ is non-trivial $t$-intersecting, then for $n\ge (t+1)(k-t+1)$, we have
\[\zeta_{k-1}(\cal F)\le \max\{\zeta_{k-1}(\cal H(n, k, t)),\zeta_{k-1}(\cal A(n, k, t))\},\]
equality holds if and only if $\cal F\cong \cal H(n, k, t)$ or $\cal F\cong \cal A(n, k, t)$.
\end{theo}

{\bf Remark.}  Suppose that $\cal F \subseteq \binom{[n]}{k}$ is a non-trivial $t$-intersecting family with maximality of number of copies of $\cal P_2$.
As new sets are added to  $\cal F$, the number of copies of $\cal P_2$ of the resulting family will not decrease. 
Hence, we can assume that $\cal F$ is maximal.
By Lemma \ref{transform}, Lemma \ref{compressed} and Lemma \ref{maximality}, we can assume that $\cal F$ is left-compressed.
As Theorem \ref{maintheo4} follows by the proof of Theorem \ref{maintheo3}, do not include the proof regarding the number of edges.

The proofs in this paper are based on the shift operator method \cite{EKR1961} and the generating set method \cite{AK}.
The definitions and fundamental properties of these methods will be introduced in the next two section. The complete proofs of our results will be provided in Sections $4$ and $5$. Finally, we will offer some concluding remarks.

\section{The left-compression operation}

Let $\mathcal F \subseteq \binom{[n]}{k}$ be a family.  For $i,j\in [n]$  and $A\in\mathcal A$, define
\[
  \delta_{ij}(A)=\begin{cases}(A\setminus\{j\})\cup\{i\}, & \hbox{if $j\in A,\ i\not\in A, (A\setminus\{j\})\cup\{i\}\not\in \mathcal{A}$;}\\A, & \hbox{otherwise,}\end{cases}
\]
and set $\Delta_{ij}(\mathcal A)=\{\delta_{ij}(A):A\in\mathcal A\}$ correspondingly. The process of  obtaining $\Delta_{ij}(\mathcal A)$ from $\mathcal A$ is called the shift operation, first introduced in \cite{EKR1961} (see also \cite{Fra1987}).

Observe  that $\Delta_{ij}(\mathcal A)$ maintains the same cardinality as  $\mathcal{A}$ and remains $k$-uniform.  A family $\mathcal A$ is called \emph{left-compressed} if $\Delta_{ij}(\mathcal A)=\mathcal A$ for all $1\leq i<j\leq n$. It is well known that if $\mathcal F$ is $t$-intersecting, then
 $\Delta_{ij}(\mathcal F)$ is so.

Similar to the classic extremal problems on $t$-intersecting families,
 one might expect that the left-compression algorithm would also be effective for addressing the $\ell_2$-norm problems in $t$-intersecting families.
Fortunately, this expectation is indeed justified. We begin by presenting a useful reformulation of the codegree squared sum for a $k$-uniform hypergraph, adapted from \cite{BL}. For completeness, we provide a concise proof.
\begin{lemma}\label{transform}
Let $\mathcal F \subseteq \binom{[n]}{k}$ be a family. Then
\[{\rm co}_2(\mathcal F)=k|\cal F|+2|\{\{F,G\}:F,G\in \cal F, |F\cap G|=k-1\}|.\]
 \end{lemma}
\begin{proof}
\begin{eqnarray*}
{\rm co}_2(\mathcal F)=\sum_{E\in \binom{[n]}{k-1}}d(E)^2&=&\sum_{E\in \binom{[n]}{k-1}}d(E) + 2\sum_{E\in \binom{[n]}{k-1}}\binom {d(E)}2\\
&=& k|\cal F|+2\sum_{E\in \binom{[n]}{k-1}}|\{\{F,G\}:F,G\in \cal F, F\cap G=E\}|\\
&=& k|\cal F|+ 2|\{\{F,G\}:F,G\in \cal F, |F\cap G|=k-1\}|.
\end{eqnarray*}
\end{proof}
\qed

We now introduce the following pivotal lemma, which serves as the foundation for our application of the classic generating set method.

\begin{lemma}\label{compressed}
If $ \mathcal A \subseteq \binom{[n]}{k} $ is $t$-intersecting, then for every $1\le i<j\le n$, we have
\[ \zeta_{k-1}(\Delta_{ij}(\mathcal A))\ge \zeta_{k-1}(\mathcal A).\]
 \end{lemma}
 \begin{proof} Let us denote $\Omega(\mathcal A)$ as the collection of ordered pairs of sets in  $\mathcal A$  that intersect in exactly $k-1$ elements, i.e.,
  \[\Omega(\mathcal A)=\{(F,G):F,G\in \cal A, |F\cap G|=k-1\}.\]
 We aim to demonstrate that  $|\Omega(\mathcal A)|\le |\Omega(\Delta_{ij}(\mathcal A))|$.

 Consider any  $F,G\in \mathcal A$ with $|F\cap G|=k-1$.
We examine how $F$ and $G$ change under the left-compression operation $\Delta_{ij}$.
 First, observe that if  $F\notin \Delta_{ij}(\mathcal A)$ and $G \in \Delta_{ij}(\mathcal A)$, then $\{i,j\}\cap G \neq \{i\}$.
 To see why, assume for contradiction that  $\{i,j\}\cap G=\{i\}$.
 Since $F\notin \Delta_{ij}(\mathcal A)$, it follows that  $j\in F, i\notin F$ and $F\cup \{i\}\setminus \{j\} \notin \cal A$.
 Combining with $|F\cap G|=k-1$ and $\{i,j\}\cap G=\{i\}$, we deduce that $G=F\cup \{i\}\setminus \{j\}\in \cal A$, which contradicts $F\cup \{i\}\setminus \{j\}\notin \cal A$.
 Similarly, if $G\notin \Delta_{ij}(\mathcal A)$ and $F \in \Delta_{ij}(\mathcal A)$, then $\{i,j\}\cap F\neq \{i\}$.
We now define the following conditions for  $F$ and $G$:

 (C1) $F\in \Delta_{ij}(\mathcal A)$, $G \in \Delta_{ij}(\mathcal A)$;

 (C2) $F\notin \Delta_{ij}(\mathcal A),G \in \Delta_{ij}(\mathcal A), \{i,j\}\subseteq G$ or $\{i,j\}\cap G=\emptyset$;

 (C3) $F\notin \Delta_{ij}(\mathcal A)$, $G \in \Delta_{ij}(\mathcal A)$ and $\{i,j\}\cap G=\{j\}$;

 (C4) $G\notin \Delta_{ij}(\mathcal A),F \in \Delta_{ij}(\mathcal A), \{i,j\}\subseteq F$ or $\{i,j\}\cap F=\emptyset$;

 (C5) $G\notin \Delta_{ij}(\mathcal A)$, $F \in \Delta_{ij}(\mathcal A)$ and $\{i,j\}\cap F=\{j\}$;

 (C6) $F\notin \Delta_{ij}(\mathcal A)$, $G \notin \Delta_{ij}(\mathcal A)$.

For each $(F,G)\in \Omega(\mathcal A)$, we define a mapping $\varphi((F,G))$ as follows:
 \[ \varphi((F,G))=
\begin{cases}
 (F,G), & \text{if $F,G$ satisfy (C1)};  \\
 (F\cup \{i\}\setminus \{j\},G), & \text{if $F,G$ satisfy (C2)};  \\
 (F,G\cup \{i\}\setminus \{j\}), & \text{if $F,G$ satisfy (C4)};  \\
(F\cup \{i\}\setminus \{j\},G\cup \{i\}\setminus \{j\}), & \text{if $F,G$ satisfy (C3) or (C5) or (C6)}.
\end{cases} \]
On the one hand, it is straightforward to verify that $\varphi(\Omega(\mathcal A))\subseteq \Omega(\Delta_{ij}(\mathcal A))$.
On the other hand, since for every pair $F,G\in \mathcal A$ with $|F\cap G|=k-1$ satisfy exactly one of the conditions (C1)--(C6), it is also clear that $\varphi$ is a injection.
Therefore, $\varphi$ is an injection from $\Omega(\mathcal A)$ to $\Omega(\Delta_{ij}(\mathcal A))$.
Consequently,  we have  $|\Omega(\mathcal A)|\le |\Omega(\Delta_{ij}(\mathcal A))|$,  completing the proof.
\qed
 \end{proof}

The following corollary is a consequence of Lemma \ref{transform}, Lemma \ref{compressed} and the fact $|\Delta_{ij}(\mathcal A)|=|\mathcal A|$.
\begin{coro}\label{compressed0}
If $ \mathcal A \subseteq \binom{[n]}{k} $ is $t$-intersecting, then for every $1\le i<j\le n$, we have
\[ {\rm co}_2(\Delta_{ij}(\mathcal A))\ge {\rm co}_2(\mathcal A).\]
In particular, there exists a left-compressed $t$-intersecting $ \mathcal B \subseteq \binom{[n]}{k} $ such that ${\rm co}_2(\mathcal B)\ge {\rm co}_2(\mathcal A)$.
 \end{coro}

\begin{lemma}[Ahlswede--Khachatrian, \cite{AK1996}]\label{maximality}
Let $t,k,n$ be positive integers such that $2\le t\leq k\leq n$.
Then there exists a left-compressed, non-trivial t-intersecting family $\cal F \subseteq \binom{[n]}{k}$ that attains the maximal possible size.
 \end{lemma}

 As new sets are added to  $\cal F$, the codegree squared sum of the resulting family will not decrease.  The following corollary is a direct consequence of Lemma \ref{compressed} and Lemma \ref{maximality}
\begin{coro}\label{compressed1}
Let $t,k,n$ be positive integers such that $2\le t\leq k\leq n$.
Then there exists a left-compressed non-trivial $t$-intersecting family $\cal F \subseteq \binom{[n]}{k}$  with maximality of codegree squared sum.
 \end{coro}

\section{The generating set}

For a subset $E$ of $[n]$, we define $\mathcal{U}(E)=\{A\subseteq [n]: E\subseteq A \}$ and $s^+(E)=\max\{i: i\in E\}$. For $\mathcal{E}\subseteq 2^{[n]}$, we set $\mathcal{U}(\mathcal{E})=\cup_{E\in \mathcal{E}}\mathcal{U}(E)$ and $s^+(\mathcal E)=\max\{s^+(E): E\in\mathcal E\}$ accordingly. Let $\mathcal{F}\subseteq\binom{[n]}{k}$. We say that a family  $g(\mathcal{F})\subseteq \bigcup_{i\leq k}\binom{[n]}{i}$   is a \emph{generating set} of $\mathcal{F}$ if $\mathcal{U}(g(\mathcal F))\cap  \binom{[n]}{k}=\mathcal{F}$. It is clear that $\mathcal{F}$ is a generating set of itself. The set of all generating sets of $\mathcal{F}$ forms a nonempty  collection, which we denote by $G(\mathcal{F})$. The concept  of generating sets of a $k$-uniform family was first introduced in \cite{AK}. For $g(\mathcal F)\in G(\mathcal F)$, let $g_*(\mathcal F)$ be the set of all minimal (in the sense of set-theoretical inclusion) elements of $g(\mathcal F)$.
Thus $g_*(\cal F)$ forms an anti-chain under the inclusion relation.
We denote $s=\min \{ s^+(g(\mathcal F)): g(\mathcal F) \in G(\mathcal F) \}$, $\cal F|_{[s]}=\{F\cap [s]: F\in \cal F\}$ and $G_*(\mathcal F)=\{g(\mathcal F)\in G(\mathcal F): g(\mathcal F)=g_*(\mathcal F), s^+(g(\mathcal F))=s\}$.
Note that $\cal F|_{[s]} \in G(\mathcal F)$, and $g(\cal F)\subseteq \cal F|_{[s]}$ for all $g(\cal F)\in G_*(\mathcal F)$.
For $g(\mathcal F)\in G_*(\mathcal F)$,  set $g^*(\mathcal F)=\{E\in g(\mathcal F): s \in E\}$, $g^*_i(\mathcal F)=\{E\in g^*(\mathcal F): |E|=i\}$  and $g^*_i(\mathcal F)'=\{E\setminus\{s\} : E\in g^*_i(\mathcal F)\}$ for $t\leq i\leq s$. 
For $E\in g(\mathcal F)$, set
\[\mathcal D(E)=\{B\in\binom{[n]}{k}: B\cap [s]=E\}\:{\rm and } \: \mathscr D(E)=\{B\in\binom{[n]}{k}: B\cap [s^+(E)]=E\}.\]
We remark that $\mathcal D(E) \subseteq \mathscr D(E)$, and $\mathcal D(E) = \mathscr D(E)$ only if $s\in E$.
For  $\mathcal E \subseteq g(\mathcal F)$,
denote $\mathcal D(\mathcal E)=\cup_{E\in \mathcal E}\mathcal D(E)$ and $\mathscr D(\mathcal E)=\cup_{E\in \mathcal E}\mathscr D(E)$.
To gain some intuitive understanding of these definitions, let's consider a simple example of \[ \mathcal F=\cal A(n, k, t)=\{F\subseteq \binom{[n]}{k}: |F\cap [t+2]|\ge t+1\}.\]
We have $g(\mathcal F)=\binom{[t+2]}{t+1}$, $s=t+2$, $g^*_{t+1}(\mathcal F)=\{F\cup \{t+2\}: F\in \binom{[t+1]}{t}\}$, $g^*_i(\mathcal F)=\emptyset$ for $i\neq t+1$, $g^*_{t+1}(\mathcal F)'=\binom{[t+1]}{t}$, $\mathcal D([t+1])=\{B\in\binom{[n]}{k}: B\cap [t+2]=[t+1]\} $ and $ \mathscr D([t+1])=\{B\in\binom{[n]}{k}: [t+1]\subset B\}$.

For an intersecting family  $\mathcal{F}\subseteq\binom{[n]}{k}$,
we say  $\cal F$ is maximal if $\cal F \cup \{E\}$ is not intersecting for any $E\in \binom{[n]}{k} \setminus \cal F$.
From \cite{AK}, we know that the generating sets  possess  the following properties.

\begin{lemma}[\cite{AK}]\label{gtf} 
Let $n,k$ and $t$ be positive integers with $k \ge t$ and $n>2k-t$.
Let $\mathcal F$ be a maximal left-compressed $t$-intersecting subfamily of $\binom{[n]}{k}$, $g(\mathcal F)\in G_*(\mathcal F)$ and  $s=s^+(g(\mathcal F))$ such that $s$ is minimum. Then the following five statements hold.

  {\rm (i)}  $|E_1\cap E_2|\geq t$ holds for all $E_1,E_2\in g(\mathcal F)$.

  {\rm (ii)} For $1\leq i<j\leq s$  and $E\in g(\mathcal F)$, we have  $F\subseteq \delta_{ij}(E)$ for some $F\in g(\mathcal F)$.

  {\rm(iii)} $\mathcal F$ is a disjoint union \[\mathcal F=\bigcup_{E\in g(\mathcal F)}\mathscr D(E).\]
                   
  {\rm(iv)} If $g^*_i(\mathcal F)\neq \emptyset$, then $g^*_{s+t-i}(\mathcal F)\neq \emptyset$ and for any $E_1\in g^*_i(\mathcal F)$, there exists $E_2\in g^*_{s+t-i}(\mathcal F)$ with $|E_1\cap E_2|=t$ and $E_1\cup E_2=[s]$, which yields $|E_1|+|E_2|=s+t$.

{\rm(v)} If  $g^*_i(\mathcal F)\neq\emptyset$ with $i\neq \frac{s+t}{2}$, then $\mathcal F_1=\mathcal F\cup \mathcal D(g^{*}_i(\mathcal F)')\backslash \mathcal D(g^{*}_{s+t-i}(\mathcal F))$ is also a $t$-intersecting subfamily of $\binom{[n]}{k}$ with
\[|\mathcal F_1|=|\mathcal F|+|g^{*}_i(\mathcal F)|\binom{n-s}{k-i+1}-|g^{*}_{s+t-i}(\mathcal F)|\binom{n-s}{k+i-s-t}.\]
 \end{lemma}

\medskip Given families $\cal F, \cal G\subseteq 2^{[n]}$ and integers  $1\le u,v \le n$, we  define
\[\zeta_{u}(\cal F,\cal G)=|\{\{F,G\}:F \in \cal F, G\in \cal G, |F|=|G|=u+1, |F\cap G|=u\}|\]
and
\[\zeta_{u,v}(\cal F,\cal G)=|\{\{F,G\}\in \zeta_{u}(\cal F,\cal G), v\in F\cap G \}|.\]
For simplicity, we denote $\zeta_{u}(\cal F)=\zeta_{u}(\cal F, \cal F)$.
We now present the following useful lemma.

\begin{lemma}\label{norm}
Let $\mathcal F\subseteq \binom{[n]}{k}$ be a left-compressed $t$-intersecting family, $g(\mathcal F)\in G_*(\mathcal F)$ and  $s=s^+(g(\mathcal F))$.
Suppose that  $g^*_i(\mathcal F)\neq\emptyset$ and $\mathcal F_1=\mathcal F\cup\mathscr D(g^{*}_i(\mathcal F)')\backslash \mathscr D(g^{*}_{j}(\mathcal F))$, where $j=s+t-i\neq i$.
Then
\begin{small}
\begin{eqnarray*}
\zeta_{k-1}(\cal F_1)-\zeta_{k-1}(\cal F)&\ge& |g_{i}^{*}(\mathcal F)|\binom{n-s}{k-i}\binom{n-s-k+i}2
+\zeta_{i-1,s}(g^{*}_{i}(\mathcal F),\cal F|_{[s]}) \binom{n-s}{k-i+1}+\\
&&|g_{i}^{*}(\mathcal F)|(s-i+1)(k-i+1)\binom{n-s}{k-i+1}
-|g_{j}^{*}(\mathcal F)|\binom{n-s}{k-j-1}\binom{n-s-k+j+1}2\\
&&-\zeta_{j-1,s}(g^{*}_{j}(\mathcal F),\cal F|_{[s]}) \binom{n-s}{k-j}
-|g^{*}_{j}(\mathcal F)|(s-j) (k-j+1)\binom{n-s}{k-j}.
\end{eqnarray*}
\end{small}
 \end{lemma}

 \begin{proof}
On the one hand, we consider the  cases $\{F_1,F_2\}$ that will increase the value of $\zeta_{k-1}(\cal F_1)-\zeta_{k-1}(\cal F)$. In such cases, at least one of $F_1$ or $F_2$ must be  contained in $g^{*}_{i}(\mathcal F)'\times \binom{[s+1,n]}{k-i+1}$.
Without loss of generality assume that $F_1 \in g^{*}_{i}(\mathcal F)'\times \binom{[s+1,n]}{k-i+1}$.
Since $g(\cal F)$ is an anti-chain under the inclusion relation, $F_2\cap [s-1]$ can not be a proper subset of $F_1\cap [s-1]$. Therefore it must be either $|F_2\cap [s-1]|=i-1$ or $|F_2\cap [s-1]|=i$.

Case 1:  $|F_2\cap [s-1]|=i-1$. Then it must be either $F_1\cap [s-1]=F_2\cap [s-1]$ or $|(F_1\cap [s-1]) \cap (F_2\cap [s-1])|=i-2$.

Subcase 1.1: $F_1\cap [s-1]=F_2\cap [s-1]$. In this case,  $|(F_1\cap [s+1,n]) \cap (F_2\cap [s+1,n])|=k-i$, The number of such $F_1\cap [s-1]$ is $|g_{i}^{*}(\mathcal F)'|=|g_{i}^{*}(\mathcal F)|$.
If $s\in F_2$, then $F_2\cap [s+1,n]$ is a proper subset of $ F_1\cap [s+1,n]$ with $|F_1\cap [s+1,n]|=|F_2\cap [s+1,n]|+1$. Thus, the number of such $\{F_1\cap [s+1,n], F_2\cap [s+1,n]\}$ is $\binom{n-s}{k-i+1}(k-i+1)$.
Otherwise $s\notin F_2$, so $\{F_1\cap [s+1,n], F_2\cap [s+1,n]\}\subseteq \binom{[s+1,n]}{k-i+1}$. Thus the number of such $\{F_1\cap [s+1,n], F_2\cap [s+1,n]\}$ is $\binom{n-s}{k-i}\binom{n-s-k+i}2$.
Therefore, the total number of such $\{F_1, F_2\}$ with $F_1\cap [s-1]=F_2\cap [s-1]$  is
\begin{eqnarray}\label{no1}
|g_{i}^{*}(\mathcal F)|(\binom{n-s}{k-i+1}(k-i+1)+\binom{n-s}{k-i}\binom{n-s-k+i}2).
\end{eqnarray}

Subcase 1.2: $|(F_1\cap [s-1]) \cap (F_2\cap [s-1])|=i-2$.
Note that there exist some $G\in g(\cal F)$ such that $G\subseteq F_2\cap [s-1]$, implying $(F_2\cap [s-1])\cup \{s\} \in \cal F|_{[s]}$.
 Additionally,  $(F_1\cap [s-1])\cup \{s\} \in g_{i}^{*}(\mathcal F)$, implying $|((F_1\cap [s-1])\cup \{s\}) \cap ((F_2\cap [s-1])\cup \{s\})|=i-1$.
Thus the number of such $\{F_1\cap [s-1],F_2\cap [s-1]\}$ with $|(F_1\cap [s-1]) \cap (F_2\cap [s-1])|=i-2$ is at least $\zeta_{i-1,s}(g_{i}^{*}(\mathcal F),\cal F|_{[s]})$.
Furthermore, $|(F_1\cap [s-1]) \cap (F_2\cap [s-1])|=i-2$ implies $F_1\cap [s+1,n] = F_2\cap [s,n] \subseteq \binom{[s+1,n]}{k-i+1}$, and the number of such pairs is $\binom{n-s}{k-i+1}$.
Hence the number of such $\{F_1, F_2\}$ with $|(F_1\cap [s-1]) \cap (F_2\cap [s-1])|=i-2$ is at least
\begin{eqnarray}\label{no2}
\zeta_{i-1,s}(g_{i}^{*}(\mathcal F),\cal F|_{[s]})\binom{n-s}{k-i+1}.
\end{eqnarray}

Case 2: $|F_2\cap [s-1]|=i$. it must be $F_1\cap [s-1] \subseteq F_2\cap [s-1]$ with $|F_2\cap [s-1]|=|F_1\cap [s-1])|+1$.
Thus, the number of such $\{F_1\cap [s-1], F_2\cap [s-1]\}$  is  $|g_{i}^{*}(\mathcal F)'|((s-1)-(i-1))=|g_{i}^{*}(\mathcal F)|(s-i)$.
Additionally, $|F_2\cap [s-1]|=i$ implies $F_2\cap [s,n] \subseteq F_1\cap [s+1,n]$ with $|F_1\cap [s+1,n] |=|F_2\cap [s,n]|+1$.
Thus, the number of such $\{F_1\cap [s+1,n], F_2\cap [s,n]\}$ is $\binom{n-s}{k-i+1}(k-i+1)$.
Hence, the number of such $\{F_1, F_2\}$ with $|F_2\cap [s-1]|=i$ is
\begin{eqnarray}\label{no3}
|g_{i}^{*}(\mathcal F)|(s-i)\binom{n-s}{k-i+1}(k-i+1).
\end{eqnarray}

On the other hand, we consider the  cases $\{F_1,F_2\}$ that will decrease the number of $\zeta_{k-1}(\cal F_1)-\zeta_{k-1}(\cal F)$.  In such cases, at least one of $F_1,F_2$ is contained in $g^{*}_{j}(\mathcal F)\times \binom{[s+1,n]}{k-i+1}$.
Without loss of generality, assume that $F_1 \in g^{*}_{j}(\mathcal F)\times \binom{[s+1,n]}{k-i+1}$.
Since $g(\cal F)$ is an anti-chain under the inclusion relation, $F_2\cap [s]$ can not be a proper subset of $F_1\cap [s]$. Thus, $|F_2\cap [s]|=j$ or $j+1$.

Case 1: $|F_2\cap [s]|=j$.  This implies that $|(F_1\cap [s]) \cap (F_2\cap [s])|=j-1$ or $F_1\cap [s]=F_2\cap [s]$.

Subcase 1.1:  $|(F_1\cap [s]) \cap (F_2\cap [s])|=j-1$.
Similar to the previous case, there exist some $G\in g(\cal F)$ such that $G\subseteq F_2\cap [s]$, which implies that $F_2\cap [s] \in \cal F|_{[s]}$.
If $s\in F_2$, then the number of such $\{F_1\cap [s],F_2\cap [s]\}$  is at most $\zeta_{j-1,s}(g_{j}^{*}(\mathcal F),\cal F|_{[s]})$.
Otherwise, $s\notin F_2$, and  $F_2=(F_1\cap [s-1])\cup \{\ell\}$ for some $\ell\in [s]\setminus F_1$. 
Thus,  the number of such $\{F_1\cap [s],F_2\cap [s]\}$  is $|g^{*}_{j}(\mathcal F)|(s-j)$.

Subcase 1.2:   $F_1\cap [s]=F_2\cap [s]$. Then $F_1\cap [s+1,n]\subseteq \binom{[s+1,n]}{k-j}$ and $F_2\cap [s+1,n] \subseteq \binom{[s+1,n]}{k-j}$ with $|(F_1\cap [s+1,n]) \cap (F_2\cap [s+1,n])|=k-j-1$.
Thus, the number of such $\{F_1, F_2\}$ is $|g_{j}^{*}(\mathcal F)|\binom{n-s}{k-j-1}\binom{n-s-k+j+1}2$.
Therefore, the total number of such $\{F_1, F_2\}$ with $|F_2\cap [s]|=j$  is at most
\begin{eqnarray}\label{no4}
(\zeta_{j-1,s}(g_{j}^{*}(\mathcal F),\cal F|_{[s]})+|g^{*}_{j}(\mathcal F)|(s-j))\binom{n-s}{k-j}+|g_{j}^{*}(\mathcal F)|\binom{n-s}{k-j-1}\binom{n-s-k+j+1}2.
\end{eqnarray}

Case 2:  $|F_2\cap [s]|=j+1$. This implies $F_1\cap [s] \subseteq F_2\cap [s]$ and $F_2\cap [s+1,n] \subseteq F_1\cap [s+1,n]$ with $|F_2\cap [s] |=|F_1\cap [s] |+1$ and $|F_1\cap [s+1,n] |=|F_2\cap [s+1,n] |+1$. Thus, the number of such pairs $\{F_1, F_2\}$ is
\begin{eqnarray}\label{no5}
|g^{*}_{j}(\mathcal F)|(s-j) (k-j)\binom{n-s}{k-j}.
\end{eqnarray}

By summing the counts in (\ref{no1}), (\ref{no2}) and (\ref{no3}), and subtracting the counts in  (\ref{no4}) and (\ref{no5}), we obtain the desired inequality.
 \end{proof}
\qed

\begin{lemma}\label{norm2}
Let $\mathcal F\subseteq \binom{[n]}{k}$ be a left-compressed $t$-intersecting family, $g(\mathcal F)\in G_*(\mathcal F)$ and  $s=s^+(g(\mathcal F))$.
Suppose that  $g(\mathcal F)\subseteq \binom{[s]}{i}$.
Let $f_q\subseteq g_i^*=g^*_i(\mathcal F)$, $f_q'=\{F\setminus \{s\}: F\in f_q\}$ and $\mathcal F_3=\mathcal F\cup\mathscr D(f_q')\backslash \mathscr D(g^{*}_{i}(\mathcal F))$.
Then
\begin{small}
\begin{eqnarray*}
&&\zeta_{k-1}(\cal F_3)-\zeta_{k-1}(\cal F)\ge |f_q|\binom{n-s}{k-i}\binom{n-s-k+i}2 +|f_q|(s-i+1)\binom{n-s}{k-i+1}(k-i+1)\\
&&-(|g_{i}^{*}|-|f_q|)(\binom{n-s}{k-i-1}\binom{n-s-k+i+1}2+(s-i)i\binom{n-s}{k-i}+(s-i)(k-i)\binom{n-s}{k-i}).
\end{eqnarray*}
\end{small}
 \end{lemma}
 \begin{proof}
Similar to the proof of Lemma \ref{norm},
we consider  pairs of sets $F_1,F_2$ with $|F_1\cap F_2|=k-1$ transitioning from $\cal F$ to $\cal F_3$. We first focus on the case of increasing the number of such pairs.
Thus we can assume that $F_1\in f_q'\times \binom{[s+1,n]}{k-i+1}$.
We split it into two cases according to whether $F_1\cap [s]=F_2\cap [s]$ or not.

\:\: (i) $F_1\cap [s]=F_2\cap [s]$.
This implies that $|(F_1\cap [s+1,n])|=|(F_2\cap [s+1,n])|=k-i+1$ and $|(F_1\cap [s+1,n])\cap (F_2\cap [s+1,n])|=k-i$.  Thus,  the number of such pairs  $\{F_1,F_2\}$ is
\begin{eqnarray}\label{ineqnorm31}
|f_q|\binom{n-s}{k-i}\binom{n-s-k+i}2.
\end{eqnarray}

\:\: (ii) $F_1\cap [s]\neq F_2\cap [s]$. We consider the case where $F_1\cap [s] \subseteq F_2\cap [s]$ with $|F_2\cap [s]|=i$.
Since $\cal F$ is left-compressed, we have $(F_1\setminus \{s\})\cup \{x\} \in \cal F$ for all $x\in [s]\setminus F_1$. Each $(F_1\setminus \{s\})\cup \{x\}$ can be regard as a $F_2$ since $(F_1\setminus \{s\})\cup \{x\}\in \cal F|_{[s]}$.
Therefore, the number of such pairs  $\{F_1,F_2\}$ is at least
\begin{eqnarray}\label{ineqnorm32}
|f_q|(s-i+1)\binom{n-s}{k-i+1}(k-i+1).
\end{eqnarray}

We now consider the case of decreasing the number of such pairs. Thus, we assume that $F_1\in (g_i^*(\cal F)\setminus f_q)\times \binom{[s+1,n]}{k-i+1}$.
We split it into three cases.

\:\: (iii) $F_1\cap [s]=F_2\cap [s]$. This implies that $|(F_1\cap [s+1,n])\cap (F_2\cap [s+1,n])|=k-i-1$ and $|F_1\cap [s+1,n]|=|F_2\cap [s+1,n]|=k-i$.
The number of such pairs $\{F_1,F_2\}$ is
\begin{eqnarray}\label{ineqnorm33}
(|g_{i}^{*}|-|f_q|)\binom{n-s}{k-i-1}\binom{n-s-k+i+1}2.
\end{eqnarray}

\:\: (iv) $F_1\cap [s] \neq F_2\cap [s]$ and $|F_2\cap [s]|=i$.
This implies that $|(F_1\cap [s])\cap (F_2\cap [s])|=i-1$.
The number of such pairs $\{F_1,F_2\}$ is at most
\begin{eqnarray}\label{ineqnorm34}
(|g_{i}^{*}|-|f_q|)i(s-i)\binom{n-s}{k-i}.
\end{eqnarray}

\:\: (v) $F_1\cap [s] \neq F_2\cap [s]$ and $|F_2\cap [s]|=i+1$. This implies that $F_1\cap [s] \subseteq F_2\cap [s]$ and $F_2\cap [s+1,n]\subseteq F_1\cap [s+1,n]$.
The number of such pairs $\{F_1,F_2\}$ is
\begin{eqnarray}\label{ineqnorm35}
(|g_{i}^{*}|-|f_q|)(s-i)(k-i)\binom{n-s}{k-i}.
\end{eqnarray}

Note that for $F_1\cap [s] \neq F_2\cap [s]$, it must be $|F_2\cap [s]|\neq i-1$. Otherwise, there exists  some $A\in g(\cal F)$ such that $A \subseteq F_2\cap [s]$, implying  $A\subseteq F_1\cap [s]$, which contradicts the fact that $g(\cal F)$ is an anti-chain under the inclusion relation. Thus, we have considered all possible cases.
By summing the numbers in  (\ref{ineqnorm31}) and (\ref{ineqnorm32}), subtracting the numbers in (\ref{ineqnorm33}), (\ref{ineqnorm34}) and (\ref{ineqnorm35}), we obtain the desired inequality.
  \end{proof}
\qed

\section{Proof of Theorem \ref{maintheo1}}
Before developing the proof, we first establish key inequalities critical to the following analysis.
\subsection{Elementary inequalities}

\begin{lemma}\label{claim1}
 Let $t,k,n$ be positive integers such that $2\le t\leq k$, $t\le i,j\le k$, $s=i+j-t$ and $n\ge (t+1)(k-t+1)$. Then \[n\ge 2k+s-2i+1 \:{\rm and}\: n\ge 2k+s-2j+1.\]
\end{lemma}
\begin{proof}
By the symmetry of $i$ and $j$, it suffices to prove that $n\ge 2k+s-2i+1$.
To this end, observe that
\begin{eqnarray*}
n-(2k+s-2i+1)&\ge& (t+1)(k-t+1)-(2k+j-i-t+1)
\end{eqnarray*}
Given that $i\ge t$ and $j\le k$,  we further have \begin{eqnarray*}
n-(2k+s-2i+1)
&\ge& f(k,t):=(t+1)(k-t+1)-(3k-2t+1).
\end{eqnarray*}
Since $\frac{\partial^2{f}(k,t)}{\partial t^2}=-1$, the function $f(k,t)$ is concave in  $t$ for $2\le t\le k$.
Therefore $$f(k,t)\ge \min\{f(k,2),f(k,k)\} =0.$$ Consequently, $n\ge 2k+s-2i+1$.
\qed
\end{proof}
\medskip

\begin{lemma}\label{lemmai2}
Let $t,k,n$ be positive integers such that $2\le t\leq k$, $t\le i \le k$, $t=2i-s$ and $n\ge (t+1)(k-t+1)$.  Then
$(s-i)(n-s+1)-(s-1)(k-i+1) \ge 0.$
\end{lemma}

\begin{proof}
Notice that $n\ge (t+1)(k-t+1)\ge 2\frac{s-1}{s-t}(k-t+1)$, then $(s-t)n-(s-1)(2k-2t+2)\ge 0$, that is, $(s-t)(n-s+1)-(s-1)(2k-s-t+2)\ge 0$.
Substituting $t=2i-s$, we have $(2s-2i)(n-s+1)-(s-1)(2k-2i+2)\ge 0$, i.e., $(s-i)(n-s+1)-(s-1)(k-i+1)\ge 0$.
Hence $(2s-2i)(n-s+1)-(s-1)(2k-2i+2)\ge 0$.
\qed
\end{proof}

\begin{lemma} \label{lemmai1}
Let $n$, $k$ and $t$ be positive integers with $n\geq(t+1)(k-t+1)$ and $t\geq 2$. If $k\geq i> t$ with $(k,i,t)\neq (t+1,t+1,t)$, then
\[f(n,k,i,t)=:(n+t-i-k)(n+i-t-k+1)(i-t)-(n+t-i-k)(i-1)(k-i)-2(i-1)(i-t)k>0.\]
\end{lemma}
\begin{proof}
Suppose $(k,i,t)\neq (t+1,t+1,t)$. As $k\geq i>t$, we have $k\geq t+2$.
By the assumption that  $n\geq (t+1)(k-t+1)$ and $i\geq t+1$,  it follows that $(i-t)(2n-2k+1)-(i-1)(k-i)\geq2(i-t)t(k-t)-(i-1)(k-t)=(2(i-t)t-i+1)(k-t)>0$, hence $f(n,k,i,t)$ is an increasing function with respective to $n$.
Denote $g(k,i,t)=f((t+1)(k-t+1),k,i,t)$.
Noting
\[
\frac{\partial{g}(k,i,t)}{\partial{i}}=(t^2 - t)k^2 + (3t - 2i + 2it + t^2 - 2t^3)k - 6i^2 - 2it^2 + 8it + 2i + t^4 - 5t^2 + t + 1,
\]
and
\begin{eqnarray*}
\frac{\partial^2{g}(k,i,t)}{\partial{i}\partial{k}}
&=&(2t^2 - 2t)k + 3t - 2i + 2it + t^2 - 2t^3\\
&\ge& (2t^2 - 2t)i + 3t - 2i + 2it + t^2 - 2t^3\\
&=& (2t^2 - 2)i - 2t^3 + t^2 + 3t\\
&\ge& (2t^2 - 2)(t+1) - 2t^3 + t^2 + 3t\\
&=& 3t^2 + t - 2>0.
\end{eqnarray*}
If $i\ge t+2$, then
\begin{eqnarray*}
\frac{\partial{f}(n,k,i,t)}{\partial{i}}
&\ge& (t^2 - t)i^2 + (3t - 2i + 2it + t^2 - 2t^3)i - 6i^2 - 2it^2 + 8it + 2i + t^4 - 5t^2 + t + 1\\
&=& (t^2 + t - 8)i^2 + (- 2t^3 - t^2 + 11t + 2)i + t^4 - 5t^2 + t + 1\\
&\ge&   (t^2 + t - 8)(t+2)^2 + (- 2t^3 - t^2 + 11t + 2)(t+2) + t^4 - 5t^2 + t + 1\\
&=& 4t^2 - 3t - 27 \ge 0,
\end{eqnarray*}
where the second inequality follows from $2(t^2 + t - 8)i + (- 2t^3 - t^2 + 11t + 2)\ge 2(t^2 + t - 8)(t+1) + (- 2t^3 - t^2 + 11t + 2)=3t^2 - 3t - 14>0$.
Otherwise $i=t+1$, then $k\ge t+2$.
Similarly, pluging $k=t+2$ and $i=t+1$ into the expression of $\frac{\partial{g}(k,i,t)}{\partial{i}}$, we have
\[
\frac{\partial{g}(k,i,t)}{\partial{i}}
\ge   4t^2 - t - 7 \ge 0.
\]
To complete the proof, it suffice to verify that $g(k,t+1,t)> 0$ for $k\ge t+2$.
However,
\begin{eqnarray*}
g(k,t+1,t)&=& (t^2 - t)k^2 + (- 2t^3 + 3t^2 + 3t - 2)k + t^4 - 2t^3 - 5t^2 - t - 3\\
&\geq& (t^2 - t)(t+2)^2 + (- 2t^3 + 3t^2 + 3t - 2)(t+2) + t^4 - 2t^3 - 5t^2 - t - 3\\
&=& 4t^2 - t - 7 > 0.
\end{eqnarray*}
This completes the proof.
\qed
\end{proof}

At last, we compare the codegree squared sums of two $t$-intersecting families.
\begin{lemma}\label{lemmai=t+1}
Let $t,k,n$ be positive integers such that $t\leq k$ and $n\ge (t+1)(k-t+1)$.
Let $\cal F=\{F\in \binom{[n]}{k}: [t]\subseteq F\}$ and $\cal G=\cal A(n, k, t)=\{G\in \binom{[n]}{k}: |G\cap [t+2]|\ge t+1\}$. Then
\[{\rm co}_2(\cal G) \le  {\rm co}_2(\cal F),\]
equality holds only if  $k=t+1$ and $n=2t+1$; in this case, $\cal F \cong \mathcal{G}^c=\{[n]\setminus G:G\in \cal G\}$.
\end{lemma}
\begin{proof}
It is trivial for $k=t$. So suppose that $k\ge t+1$.
For $A\in \binom{[n]}{k-1}$, we obtain
\[d_{\cal F}(A)=
\begin{cases}
n-k+1, & \text{if $[t]\subseteq A$};  \\
 1, & \text{if $|A\cap [t]|=t-1$};  \\
0, & \text{otherwise}.
\end{cases} \]
and
\[d_{\cal G}(A)=
\begin{cases}
n-k+1, & \text{if $|A\cap [t+2]|\ge t+1$};  \\
 2, & \text{if $|A\cap [t+2]|=t$};  \\
0, & \text{otherwise}.
\end{cases} \]
Hence \begin{eqnarray}\label{co21}
{\rm co}_2(\cal F)=\sum_{A\in \binom{[n]}{k-1}}d_{\cal G}(A)^2= \binom{n-t}{k-t-1}(n-k+1)^2+t\binom{n-t}{k-t} \end{eqnarray} and
\begin{eqnarray}\label{co22}{\rm co}_2(\cal G)=((t+2) \binom{n-t-2}{k-t-2}+ \binom{n-t-2}{k-t-3})(n-k+1)^2+2(t+2)(t+1) \binom{n-t-2}{k-t-1}.\end{eqnarray}
Substituting $\binom{n-t}{k-t-1}=\binom{n-t-2}{k-t-1}+2\binom{n-t-2}{k-t-2}+\binom{n-t-2}{k-t-3}$ into (\ref{co21}), and extracting the common factor $\binom{n-t-2}{k-t-1}$, we obtain
\begin{small}
\begin{eqnarray}\label{co23} \nonumber
&&{\rm co}_2(\cal F)-{\rm co}_2(\cal G)\\
&>& (((n-k+1)-t(k-t+1))(n-k+1)- 2(t+2)(t+1)+\frac{t(n-t)(n-t-1)}{(k-t)(n-k)}) \binom{n-t-2}{k-t-1} \notag\\
&\ge&  ((t(k-t+1)-3t-3)(t+2)+\frac{t(n-t)(n-t-1)}{(k-t)(n-k)})\binom{n-t-2}{k-t-1}.
\end{eqnarray}
\end{small}
where the last inequality follows from $n\ge (t+1)(k-t+1)$.
Discussing all possible cases as following, we obtain that either RHS of (\ref{co23})$>$0 or RHS of (\ref{co21}) $\ge$ RHS of (\ref{co22}), and hence the desired inequality follows.
 \begin{itemize}
 \item $k-t+1\ge 6$, or $k-t+1\ge 5$ for $t\ge 2$, or $k-t+1\ge 4$ for $t\ge 3$; this case is obvious;
 \item $k-t+1=2$; RHS of (\ref{co21}) $\ge 2(t+2)(t+1)=$ RHS of (\ref{co22}), and equality holds only if $n=2t+2$;
 \item $k-t+1=3$; RHS of (\ref{co21}) $\ge (t+4)(2t+5)^2>2(t+2)(t+1)(2t+4)=$ RHS of (\ref{co22});
 \item $k-t+1=4$; since $n\ge 4t+4=10$, RHS of (\ref{co23}) $\ge (t-3)(t+2)+ \frac{t(3t+4)(3t+3)}{3(3t+1)}>0$ for each $t=1,2$;
 \item $k-t+1=5$ and $t=1$; similarly, RHS of (\ref{co23}) $\ge (2t-3)(t+2)+ \frac{t(4t+5)(4t+4)}{4(4t+1)}=-3+\frac{18}{5}>0$.
 \end{itemize}
  \qed
\end{proof}

Now it is ready to prove Theorem \ref{maintheo1}.

\subsection{Proof of Theorem \ref{maintheo1}}
Let $t,k,n$ be positive integers such that $1\le t\leq k\leq n$ and $n\ge (t+1)(k-t+1)$.
The case of $t=1$ was done by Brooks and Linz \cite{BL}.
Hence we may assume that $t\ge 2$.
Suppose that $\cal F \subseteq \binom{[n]}{k}$ is a $t$-intersecting family with maximality of codegree squared sum.
As new sets are added to  $\cal F$, the codegree squared sum of the resulting family will not decrease. 
Hence, $\cal F$ is maximal.

{\em Proof of the inequality of ${\rm co}_2(\cal F)\le {\binom{n-t}{k-t}}(t+(n-k+1)(k-t))$.}
By Corollary \ref{compressed0}, we can assume that $\cal F$ is left-compressed.
Let $g(\cal F)\in G_*(\cal F)$ be a generating set of $\cal F$ such that $s=s^+(g(\mathcal F))$ is minimum.
Let $i$ be an integer such that $g^*_i(\mathcal F)\neq\emptyset$.
Then $g^*_{s+t-i}(\mathcal F)\neq\emptyset$ by (iv) of Lemma \ref{gtf}.
If there is a set $A\in g(\mathcal F)$ of size $t$, then we immediately obtain that $g(\mathcal F)=\{[t]\}$  since $g(\mathcal F)$ is $t$-intersecting. Hence the desired result holds.
Similarly, if $|g(\mathcal F)|=1$, then $g(\mathcal F)=\{[t]\}$  since $g(\mathcal F)$ is $t$-intersecting and the codegree squared sum of $\cal F$ is maximum. Thus, the desired result holds as well.
So we suppose $\min_{A\in g(\mathcal A)}|A|\ge t+1$ and $|g(\mathcal F)|\ge 2$ in the sequel.
This yields that $i\ge t+1$ and $s\ge i+1$.
On the other hand, we may assume $i\leq (s+t)/2$ by symmetry; that is, $s\geq 2i-t$.
Denote $j=s+t-i$.
We divide the proof into two cases

\medskip
{\bf Case I}. $i<j$.

Set
\[
\mathcal F_1=\mathcal F\cup \mathscr D(g^{*}_{i}(\mathcal F)')\setminus \mathscr D(g^{*}_{s+t-i}(\mathcal F)) {\rm \:\: and \:\:} \mathcal F_2=\mathcal F\cup \mathscr D(g^{*}_{s+t-i}(\mathcal F)')\setminus \mathscr D(g^{*}_{i}(\mathcal F)).
\]
We know that both $\mathcal F_1$ and $\mathcal F_2$ are $t$-intersecting.
We consider the changes of codegree squared sum between $\cal F_{\ell}$ and $\cal F$ for each $\ell=1, 2$.
We will show that $|\cal F_1|+|\cal F_2|>2|\cal F|$, and $\zeta_{k-1}(\cal F_1)+\zeta_{k-1}(\cal F_2)>2\zeta_{k-1}(\cal F)$.
Then combining Lemma \ref{transform}, we have ${\rm co}_2(\mathcal F_{\ell})>{\rm co}_2(\mathcal F)$ for some $\ell\in \{1,2\}$, contradicting that ${\rm co}_2(\mathcal F)$ is maximum.
Thus the desired inequality follows.

We first consider the changes of the number of edges.
By (iii) of Lemma \ref{gtf}, we have
\begin{eqnarray}\label{ineq11}
|\cal F_1|=|\cal F|+|g_{i}^{*}(\mathcal F)|\binom{n-s}{k-i+1}-|g_{j}^{*}(\mathcal F)|\binom{n-s}{k-j},
\end{eqnarray}
and similarly
\begin{eqnarray}\label{ineq12}
|\cal F_2|=|\cal F|+|g_{j}^{*}(\mathcal F)|\binom{n-s}{k-j+1}-|g_{i}^{*}(\mathcal F)|\binom{n-s}{k-i}.
\end{eqnarray}
Combining (\ref{ineq11}) and (\ref{ineq12}), we have
\begin{eqnarray}\label{ineq13}
|\cal F_1|+|\cal F_2|-2|\cal F|=|g_{i}^{*}(\mathcal F)|\left(\binom{n-s}{k-i+1}-\binom{n-s}{k-i}\right)+|g_{j}^{*}(\mathcal F)|\left(\binom{n-s}{k-j+1}-\binom{n-s}{k-j}\right).
\end{eqnarray}
Notice that $\binom{n-s}{k-i+1}-\binom{n-s}{k-i}\ge 0$ is equivalent to $n\ge 2k+s-2i+1$.
Hence $\binom{n-s}{k-i+1}-\binom{n-s}{k-i}\ge 0$ by Lemma \ref{claim1}.
Similarly, $\binom{n-s}{k-j+1}-\binom{n-s}{k-j}\ge 0$.
Thus $|\cal F_1|+|\cal F_2|-2|\cal F|\ge 0$.

\medskip
Now we consider the changes of pairs that intersects with $k-1$ elements.
By Lemma \ref{norm},
we have

\begin{small}
\begin{eqnarray*}
\zeta_{k-1}(\cal F_1)-\zeta_{k-1}(\cal F)&\ge& |g_{i}^{*}(\mathcal F)|\binom{n-s}{k-i}\binom{n-s-k+i}2
+\zeta_{i-1,s}(g^{*}_{i}(\mathcal F),\cal F|_{[s]}) \binom{n-s}{k-i+1}+\\
&&|g_{i}^{*}(\mathcal F)|(s-i+1)(k-i+1)\binom{n-s}{k-i+1}
-|g_{j}^{*}(\mathcal F)|\binom{n-s}{k-j-1}\binom{n-s-k+j+1}2\\
&&-\zeta_{j-1,s}(g^{*}_{j}(\mathcal F),\cal F|_{[s]}) \binom{n-s}{k-j}
-|g^{*}_{j}(\mathcal F)|(s-j) (k-j+1)\binom{n-s}{k-j}.
\end{eqnarray*}
\end{small}
Similarly,
\begin{small}
\begin{eqnarray*}
\zeta_{k-1}(\cal F_2)-\zeta_{k-1}(\cal F)&\ge& |g_{j}^{*}(\mathcal F)|\binom{n-s}{k-j}\binom{n-s-k+j}2+
\zeta_{j-1,s}(g^{*}_{j}(\mathcal F),\cal F|_{[s]}) \binom{n-s}{k-j+1}+\\
&&|g_{j}^{*}(\mathcal F)|(s-j+1)(k-j+1)\binom{n-s}{k-j+1}-
|g_{i}^{*}(\mathcal F)|\binom{n-s}{k-i-1}\binom{n-s-k+i+1}2\\
&&-\zeta_{i-1,s}(g^{*}_{i}(\mathcal F),\cal F|_{[s]}) \binom{n-s}{k-i}
-|g^{*}_{i}(\mathcal F)|(s-i) (k-i+1)\binom{n-s}{k-i}.
\end{eqnarray*}
\end{small}

Notice that $\binom{n-s}{k-i}\binom{n-s-k+i}2\ge \binom{n-s}{k-i-1}\binom{n-s-k+i+1}2$  and $\binom{n-s}{k-i+1} \ge \binom{n-s}{k-i}$ are both equivalent to $n\ge 2k+s-2i+1$,
which holds by Claim \ref{claim1}.
Therefore, the following three inequalities hold.
\begin{small}
\begin{eqnarray*}
\binom{n-s}{k-i}\binom{n-s-k+i}2 &\ge& \binom{n-s}{k-i-1}\binom{n-s-k+i+1}2,\\
 \binom{n-s}{k-i+1} &\ge&  \binom{n-s}{k-i},\\
(s-i+1)(k-i+1)\binom{n-s}{k-i+1} &\ge& (s-i) (k-i+1)\binom{n-s}{k-i}.
\end{eqnarray*}
\end{small}
Similarly,
\begin{small}
\begin{eqnarray*}
\binom{n-s}{k-j}\binom{n-s-k+j}2 &>& \binom{n-s}{k-j-1}\binom{n-s-k+j+1}2,\\
\binom{n-s}{k-j+1} &\ge& \binom{n-s}{k-j},\\
(s-j+1)(k-j+1)\binom{n-s}{k-j+1} &\ge& (s-j) (k-j+1)\binom{n-s}{k-j}.
\end{eqnarray*}
\end{small}
Hence $\zeta_{k-1}(\cal F_1)+\zeta_{k-1}(\cal F_2)-2\zeta_{k-1}(\cal F)>0$.
Combining $|\cal F_1|+|\cal F_2|-2|\cal F|>0$ in (\ref{ineq13}), we have
\[
{\rm co}_2(\cal F_1)+{\rm co}_2(\cal F_2)=(k|\cal F_1|+\zeta_{k-1}(\cal F_1))+(k|\cal F_1|+\zeta_{k-1}(\cal F_1))>2(k|\cal F|+\zeta_{k-1}(\cal F)))=2{\rm co}_2(\cal F).
\]
By pigeonhole principle, we have ${\rm co}_2(\cal F_1)>{\rm co}_2(\cal F)$ or ${\rm co}_2(\cal F_2)>{\rm co}_2(\cal F)$,
contradicting to the maximality of ${\rm co}_2(\cal F)$, i.e., ${\rm co}_2(\cal F_{\ell})\le {\rm co}_2(\cal F)$ for each $\ell=1,2$.

\medskip
{\bf Case II}. $i=j$; that is, $i=j=\frac{s+t}2$.

Firstly, we consider the case $i=t+1$. Then $\cal F=\{F\in \binom{[n]}{k}: |F\cap [t+2]|\ge t+1\}$, and we are done by Lemma \ref{lemmai=t+1}. So suppose that $i\ge t+2$.
For $p \in [s-1]$, denote $f_{p}=\{F\in g_{i}^{*}(\mathcal F)': p \notin F\}$.
Thus $\cup_{\ell\in [s-1]}f_{p}=g_{i}^{*}(\mathcal F)'$ and for every $F\in g_{i}^{*}(\mathcal F)'$, $F$ belongs to exactly $s-i$ $f_{p}$'s.
By Pigeonhole Principle, there exists $q\in [s-1]$ such that
\begin{eqnarray}\label{ineq31}
|f_q|\ge \frac{s-i}{s-1}|g_{i}^{*}(\mathcal F)'|=\frac{s-i}{s-1}|g_{i}^{*}(\mathcal F)|.
\end{eqnarray}
Set
\[\mathcal F_3=\mathcal F\cup \mathcal D(f_q)\setminus \mathcal D(g_{i}^{*}(\mathcal F)).\]	
Notice that $\mathcal F_3$ is $t$-intersecting.	
We consider the changes of codegree squared sum between $\cal F_{3}$ and $\cal F$.
We will show that $|\cal F_3|>|\cal F|$, and the number of pairs of sets that intersects with $k-1$ elements in
$\cal F_3$ is larger than $\cal F$'s.
Then combining Lemma \ref{transform}, we have ${\rm co}_2(\mathcal F_{3})>{\rm co}_2(\mathcal F)$, contradicting that ${\rm co}_2(\mathcal F)$ is maximum.
Thus the desired inequality follows.

We first consider the changes of the number of edges.
It is easy to see
\begin{eqnarray*}
|\cal F_3|-|\cal F|&=&|f_q|\binom{n-s+1}{k-i+1}-|g_{i}^{*}(\mathcal F)|\binom{n-s}{k-i}\\
&\ge & \frac{s-i}{s-1}|g_{i}^{*}(\mathcal F)|\binom{n-s+1}{k-i+1}-|g_{i}^{*}(\mathcal F)|\binom{n-s}{k-i}\\
&=& |g_{i}^{*}(\mathcal F)|\binom{n-s}{k-i}\left( \frac{(s-i)(n-s+1)}{(s-1)(k-i+1)} -1 \right)>0,
\end{eqnarray*}
where the last inequality follows from Lemma \ref{lemmai2}.

Now we consider the change of the number of  pairs of sets in $\mathcal F$ that intersects $k-1$ elements.
Similar to Case 1, the number of pairs of sets that intersects with $k-1$ elements in
$\cal F_3$ minus the number of such pairs of sets in $\cal F$ is at least
\begin{small}
\begin{eqnarray*}\label{ineqi=j}
&&|f_q|\binom{n-s}{k-i}\binom{n-s-k+i}2 +|f_q|(s-i+1)\binom{n-s}{k-i+1}(k-i+1)\\
&&-(|g_{i}^{*}|-|f_q|)(\binom{n-s}{k-i-1}\binom{n-s-k+i+1}2+(s-i)i\binom{n-s}{k-i}+(s-i)(k-i)\binom{n-s}{k-i})\\
&\ge& \frac{|f_q|}{i-t}\binom{n-s}{k-i}((n+t-i-k)(n+i-t-k+1)(i-t)-(n+t-i-k)(i-1)(k-i)-2(i-1)(i-t)k)\\
&>&0,
\end{eqnarray*}
\end{small}
where the first inequality follows from $|f_q|\ge \frac{s-i}{s-1}|g_{i}^{*}(\mathcal F)|$ in (\ref{ineq31}) and $s=2i-t$, and the last inequality follows from Lemma \ref{lemmai1}.
So ${\rm co}_2(\cal F_3)>{\rm co}_2(\cal F)$, contradicting to the maximality of ${\rm co}_2(\cal F)$.
This completes the proof of the inequality.

{\em Proof of the uniqueness.}
Suppose that ${\rm co}_2(\cal F)= {\binom{n-t}{k-t}}(t+(n-k+1)(k-t))$.
It follows that $|\cal F|=\binom{n-t}{k-t}$ by the proof above.
It is well known that, up to permutations, $\mathcal{F}=\cal G_1:=\{F\in {\binom{[n]}{k}}: [t]\subseteq F\}$ for $n\ge (t+1)(k-t+1)$, and one more possibility $\cal F=\cal G_2:=\{F\in \binom{[n]}{k}: |F\cap [t+2]|\ge t+1\}$ for $n = (t+1)(k-t+1)$ (Page 126, Ahlswede-Khachatrian Theorem \cite{AK}).
By Lemma \ref{lemmai=t+1}, ${\rm co}_2(\cal G_1)>{\rm co}_2(\cal G_2)$.
Hence, $\cal G_1$ is the unique extremal family up to permutations.

\section{Proof of Theorem \ref{maintheo3}}

Firstly,  we need the following lemma in the proof.

\begin{lemma}\label{norm3}
Let $t+1\le s \le k+1$ and $\cal F_s \subseteq {[n]\choose k}$ with $g(\cal F_s)=\{[t]\cup \{\ell\}: \ell \in [t+1,s]\} \cup \{[s] \setminus \{\ell\}: \ell \in [t]\}$.
Then for each $t+1\le s \le k$, we have
\begin{small}
\begin{eqnarray*}
\zeta_{k-1}(\cal F_{s+1})-\zeta_{k-1}(\cal F_s) = \binom{n-s-1}{k-t-2}\binom{n-s-k+t+1}2+(s-t)(k-t)\binom{n-s-1}{k-t-1}\\
-t\binom{n-s-1}{k-s}\binom{n-k-1}2-\binom{t}2\binom{n-s-1}{k-s+1}-2t(k-s+1)\binom{n-s-1}{k-s+1}.
\end{eqnarray*}
\end{small}
 \end{lemma}

 \begin{proof}
  We analyze the cases that affect the difference  $\zeta_{k-1}(\cal F_{s+1})-\zeta_{k-1}(\cal F_s)$ by considering pairs $\{F_1,F_2\}$ with $|\mathcal F_1\cap \mathcal F_2|=k-1$.

We first consider the cases where the number of such pairs increases. We assume that $F_1 \in ([t]\cup \{s+1\}) \times \binom{[s+2,n]}{k-t-1}$ and $F_2\in \cal F_{s+1}$ with $|F_1\cap F_2|=k-1$.

(i) $|F_2\cap [s+1]|=t+1$ and $F_1\cap [s+1]=F_2\cap [s+1]$. Here,  $|F_1\cap F_2\cap [s+2,n]|=k-t-2$.  The number of such pairs is $\binom{n-s-1}{k-t-2}\binom{n-s-k+t+1}2$.

(ii) $|F_2\cap [s+1]|=t+1$ and $F_1\cap [s+1]\neq F_2\cap [s+1]$.
Here, $|F_1 \cap F_2\cap [s+1]|=t$ and $F_1\cap [s+2,n]=F_2\cap [s+2,n]$. The number of such pairs is $(s-t)\binom{n-s-1}{k-t-1}$.

(iii) $|F_2\cap [s+1]|=t+2$.
Then $F_1\cap [s+1]$ is a proper subset of $F_2\cap [s+1]$.
The number of such pairs is $(s-t)\binom{n-s-1}{k-t-1}(k-t-1)$.
Thus, the total number of pairs $\{F_1, F_2\}$ that increase the difference is
\begin{eqnarray}\label{s1}
\binom{n-s-1}{k-t-2}\binom{n-s-k+t+1}2+(s-t)(k-t)\binom{n-s-1}{k-t-1}.
\end{eqnarray}

We now consider the cases where the number of such pairs decreases. We assume  $F_1 \in D(\{[s] \setminus \{\ell\}: \ell \in [t]\}) \setminus D(\{[s+1] \setminus \{\ell\}: \ell \in [t]\})$ and $F_2\in \cal F_{s}$ with $|F_1\cap F_2|=k-1$.

(iv) $|F_2\cap [s+1]|=s-1$ and $F_1\cap [s+1]=F_2\cap [s+1]$.
Here, $|(F_1\cap [s+2,n]) \cap (F_2\cap [s+2,n])|=k-s$. The number of such pairs is $t\binom{n-s-1}{k-s}\binom{n-k-1}2$.

(v) $|F_2\cap [s+1]|=s-1$ and $F_1\cap [s+1] \neq F_2\cap [s+1]$.
Here, $F_1\cap [s+2,n] = F_2\cap [s+2,n]$. The number of such pairs is $\binom{t}2\binom{n-s-1}{k-s+1}$.

(vi) $|F_2\cap [s+1]|=|F_1\cap [s+1]|+1$.
Here,  $F_1\cap [s+1]$ is a proper subset of $F_2\cap [s+1]$.
The number of such pairs is $2t(k-s+1)\binom{n-s-1}{k-s+1}$.
So the total number of pairs that decrease the difference is
\begin{eqnarray}\label{s2}
t\binom{n-s-1}{k-s}\binom{n-k-1}2+\binom{t}2\binom{n-s-1}{k-s+1}+2t(k-s+1)\binom{n-s-1}{k-s+1}.
\end{eqnarray}

By subtracting the number of pairs in (\ref{s2}) from the number of pairs in (\ref{s1}), we obtain the desired inequality.
 \end{proof}
\qed

Now, we are ready to prove Theorem \ref{maintheo3}.

\medskip
{\bf Proof of Theorem \ref{maintheo3}}.
Let $t,k,n$ be positive integers such that $2\leq t\leq k\leq n$ and $n\geq(t+1)(k-t+1)$. Suppose $\mathcal{F}\subseteq\binom{[n]}{k}$ is a non-trivial $t$-intersecting family with maximal codegree squared sum. 
As new sets are added to  $\cal F$, the codegree squared sum of the resulting family will not decrease. 
Hence, $\cal F$ is maximal.

{\em Proof of the inequality of ${\rm co}_2(\cal F)\le \max\{{\rm co}_2(\cal H(n, k, t)),{\rm co}_2(\cal A(n, k, t))\}$.}
By Corollary \ref{compressed1}, we may assume that $\mathcal{F}$ is left-compressed.
Let \(g(\mathcal{F})\in G_*(\mathcal{F})\) be a generating set of \(\mathcal{F}\) such that \(s=s^+(g(\mathcal{F}))\) is minimum. Let \(i\) be the smallest integer such that \(g^*_i(\mathcal{A})\neq\emptyset\).
Clearly, $i\ge t+1$ and  $s\ge t+2$.

At first, we assume that $i\neq \frac{s+t}{2}$.  According to the proof of Theorem \ref{maintheo1},  either \(\mathcal{F}_1=\mathcal{F}\cup\mathcal{D}(g^{*}_{i}(\mathcal{F}'))\setminus\mathcal{D}(g^{*}_{s+t-i}(\mathcal{B}))\) or \(\mathcal{F}_2=\mathcal{F}\cup\mathcal{D}(g^{*}_{s+t-i}(\mathcal{F}'))\setminus\mathcal{D}(g^{*}_{i}(\mathcal{F}))\) has a larger codegree squared sum. Thus, at least one of \(\mathcal{F}_1\) and \(\mathcal{F}_2\) must be a trivial family. Without loss of generality, assume \(\mathcal{F}_1\) is trivial. Consequently, each element of \(\mathcal{F}_1\) contains $[u]$ for some \(u\ge t\).

Suppose \(u\geq t+1\). Then \(\text{co}_2(\mathcal{F}_1)<\text{co}_2(\mathcal{A}(n,k,t))\) since \(\mathcal{F}_1\) is a proper subfamily of \(\mathcal{A}(n,k,t)\) up to isomorphism, contradicting the maximality of \(\text{co}_2(\mathcal{F})\). Therefore, \(u=t\)   and \(\{[t]\cup\{\ell\}:\ell\in[t+1,s]\}\subseteq g(\mathcal{F})\).

On the one hand, since \(g(\mathcal{F})\) is minimal (in the sense of set-theoretical inclusion), there are no other sets in \(g(\mathcal{F})\setminus\{[t]\cup\{\ell\}:\ell\in[t+1,s]\}\) containing \([t]\). On the other hand, since \(g(\mathcal{F})\) is \(t\)-intersecting, \(|F\cap[t]|=t-1\) and \([t+1,s]\subseteq F\) for each \(F\in g(\mathcal{F})\setminus\{[t]\cup\{\ell\}:\ell\in[t+1,s]\}\). By the maximality of \(\text{co}_2(\mathcal{F})\), we obtain
$g(\mathcal{F})\setminus\{[t]\cup\{\ell\}:\ell\in[t+1,s]\}=\{[s]\setminus\{\ell\}:\ell\in[t]\}$,
yielding
\[g(\mathcal{F})=\{[t]\cup\{\ell\}:\ell\in[t+1,s]\}\cup\{[s]\setminus\{\ell\}:\ell\in[t]\}\]
for some $s \in [t+2,k+1]$.
Let $\cal F_s \subseteq {[n]\choose k}$ such that $g(\cal F_s)=\{[t]\cup \{\ell\}: \ell \in [t+1,s]\} \cup \{[s] \setminus \{\ell\}: \ell \in [t]\}$.
By Lemma \ref{norm3}, for each \(s\in[t+2,k]\), we have
\[
\begin{aligned}
\zeta_{k-1}(\mathcal{F}_{s+1})&-\zeta_{k-1}(\mathcal{F}_s)=\binom{n-s-1}{k-t-2}\binom{n-s-k+t+1}{2}+(s-t)(k-t)\binom{n-s-1}{k-t-1}\\
&-t\binom{n-s-1}{k-s}\binom{n-k-1}{2}-\binom{t}{2}\binom{n-s-1}{k-s+1}-2t(k-s+1)\binom{n-s-1}{k-s+1}.
\end{aligned}
\]
It suffices to prove that $\zeta_{k-1}(\cal F_{s-1})\ge \zeta_{k-1}(\cal F_s)$ or $\zeta_{k-1}(\cal F_{s+1})\ge \zeta_{k-1}(\cal F_s)$ for each $s\in [t+3,k]$.
For the contrary, we suppose that there exists some $s\in [t+3,k]$ such that $\zeta_{k-1}(\cal F_{s-1})< \zeta_{k-1}(\cal F_s)$ and $\zeta_{k-1}(\cal F_{s+1})< \zeta_{k-1}(\cal F_s)$.
As $\zeta_{k-1}(\cal F_{s-1})< \zeta_{k-1}(\cal F_s)$, we have
\begin{eqnarray*}
 &&\binom{n-s-1}{k-t-2}\binom{n-s-k+t+1}2+(s-t)(k-t)\binom{n-s-1}{k-t-1}\\&&<
t\binom{n-s-1}{k-s}\binom{n-k-1}2+\binom{t}2\binom{n-s-1}{k-s+1}+2t(k-s+1)\binom{n-s-1}{k-s+1}.
\end{eqnarray*}
Similarly, as $\zeta_{k-1}(\cal F_{s-1})< \zeta_{k-1}(\cal F_s)$, we have
\begin{eqnarray*}
 &&t\binom{n-s}{k-s+1}\binom{n-k-1}2+\binom{t}2\binom{n-s}{k-s+2}+2t(k-s+2)\binom{n-s}{k-s+2}\\&&<
\binom{n-s}{k-t-2}\binom{n-s-k+t+2}2+(s-1-t)(k-t)\binom{n-s}{k-t-1}.
\end{eqnarray*}
Combining the above two inequalities, we obtain
\begin{eqnarray*}
 \frac{\binom{n-s-1}{k-t-2}\binom{n-s-k+t+1}2+(s-t)(k-t)\binom{n-s-1}{k-t-1}}{\binom{n-s}{k-t-2}\binom{n-s-k+t+2}2+(s-1-t)(k-t)\binom{n-s}{k-t-1}}<
\frac{t\binom{n-s-1}{k-s}\binom{n-k-1}2+\binom{t}2\binom{n-s-1}{k-s+1}+2t(k-s+1)\binom{n-s-1}{k-s+1}}{t\binom{n-s}{k-s+1}\binom{n-k-1}2+\binom{t}2\binom{n-s}{k-s+2}+2t(k-s+2)\binom{n-s}{k-s+2}}.\end{eqnarray*}
On the one hand, as
\begin{eqnarray*}
\frac{\binom{n-s-1}{k-t-2}\binom{n-s-k+t+1}2}{\binom{n-s}{k-t-2}\binom{n-s-k+t+2}2}=\frac{n-s-k+t}{n-s}<\frac{(s-t)(n-s-k+t+1)}{(s-1-t)(n-s)}=\frac{(s-t)(k-t)\binom{n-s-1}{k-t-1}}{(s-1-t)(k-t)\binom{n-s}{k-t-1}},
\end{eqnarray*}
we have
\begin{eqnarray*}
 \frac{\binom{n-s-1}{k-t-2}\binom{n-s-k+t+1}2+(s-t)(k-t)\binom{n-s-1}{k-t-1}}{\binom{n-s}{k-t-2}\binom{n-s-k+t+2}2+(s-1-t)(k-t)\binom{n-s}{k-t-1}}
 >\frac{n-s-k+t}{n-s}.
 \end{eqnarray*}
On the other hand,
\begin{eqnarray*}
\frac{t\binom{n-s-1}{k-s}\binom{n-k-1}2}{t\binom{n-s}{k-s+1}\binom{n-k-1}2}=\frac{2t(k-s+1)\binom{n-s-1}{k-s+1}}{2t(k-s+2)\binom{n-s}{k-s+2}}=\frac{k-s+1}{n-s}<\frac{k-s+2}{n-s}=\frac{\binom{t}2\binom{n-s-1}{k-s+1}}{\binom{t}2\binom{n-s}{k-s+2}},
\end{eqnarray*}
we have
\begin{eqnarray*}
\frac{t\binom{n-s-1}{k-s}\binom{n-k-1}2+\binom{t}2\binom{n-s-1}{k-s+1}+2t(k-s+1)\binom{n-s-1}{k-s+1}}{t\binom{n-s}{k-s+1}\binom{n-k-1}2+\binom{t}2\binom{n-s}{k-s+2}+2t(k-s+2)\binom{n-s}{k-s+2}}
< \frac{k-s+2}{n-s}.
 \end{eqnarray*}
Therefore,
\begin{eqnarray*}
\frac{n-s-k+t}{n-s}<\frac{k-s+2}{n-s}, \text{ i.e.,} \: n<2k-t+2.
\end{eqnarray*}
However, since $n\ge (t+1)(k-t+1)$, $n-(2k-t+2)\ge f(k,t):=(t+1)(k-t+1)-(2k-t+2)$.
As $f(k,t)$ is concave on $t$ and $2\le t \le k-1$, $f(k,t) \ge \min \{f(k,2),f(k,k-1)\}=k-3>0$, a contradiction.

At last, we assume that $i=\frac{s+t}{2}$. It suffices to prove that $s=t+2$, since \(\mathcal{F}=\mathcal A(n,k,t)\) in this case clearly.
We claim that $s=t+2$. For the contrary suppose that $s>t+2$. Since $i=\frac{s+t}{2}$ is an integer, we know that $s\ge t+4$. 
According to the proof of Theorem \ref{maintheo1},  $\mathcal F_3=\mathcal F\cup\mathscr D(f_q)\backslash \mathscr D(g^{*}_{i}(\mathcal F))$ has a larger codegree squared sum  for some $q\in [s-1]$, where  $f_{q}=\{F\in g_{i}^{*}(\mathcal F)': q \notin F\}$. Thus, \(\mathcal{F}_3\) must be a trivial family by the assumption.  
Let $G,F\in g_{i}^{*}(\mathcal F)'$ such that $G\cup F=[s-1]$ and $|G\cap F|=t-1$.  
Since $s\ge t+4$, $|[s-1]\setminus G|\ge 2$ and $|[s-1]\setminus F|\ge 2$.
Let $a,b\in [s-1]\setminus G$ and $c,d \in [s-1]\setminus G$ with $a\neq b$ and $c\neq d$.
By term (ii) and (iii) of Lemma \ref{gtf}, we have $ \mathscr D(E) \subset \mathcal F$ for all $E\in \{G\cup \{a\}, G\cup \{b\}, F\cup \{c\}, F\cup \{d\}\}$.
However, it is easy to see that  $\cap_{H\in \mathcal F_3}H\subset (G\cup \{a\})\cap( G\cup \{b\})\cap( F\cup \{c\})\cap (F\cup \{d\})=G\cap F$, contradicting that $\mathcal F_3$ is trivial.

{\em Proof of the uniqueness.}
Suppose that ${\rm co}_2(\cal F)=\max\{{\rm co}_2(\cal H(n, k, t)),{\rm co}_2(\cal A(n, k, t))\}$.
It follows that $|\cal F|=\max\{|\cal H(n, k, t)|,|\cal A(n, k, t)|\}$ by the proof above.
It is well known that, up to permutations, $\mathcal{F}=\cal H(n, k, t)$ or $\cal F=\cal H(n, k, t)$ (Page 124, Ahlswede-Khachatrian Theorem \cite{AK1996}).
\qed

\section{Concluding remarks}
In the present paper,  we prove an Erd\H{o}s--Ko--Rado Theorem in $\ell_2$-norm.
As a by-product, we prove a Frankl--Hilton--Milner Theorem for $t\ge 2$.
It is interesting to consider the case of $t=1$, which was conjectured by Brooks and  Linz \cite{BL}.

\begin{con}[Brooks and  Linz, \cite{BL}]\label{BLc}
Let $ \mathcal F \subseteq \binom{[n]}{k} $ be a non-trivial intersecting family.
Then, for $k\ge 3$ and $n>2k$,
\[{\rm co}_2(\cal F)\le {\rm co}_2(\cal H(n,k,1)),\]
with equality if and only if $\mathcal{F} \cong H(n,k,1)$ if $k\ge 4$,
and if and only if $\mathcal{F} \cong H(n,k,1)$ or $\mathcal{F} \cong A(n,k,1)$ if $k=3$.
\end{con}

Using the methods presented in this paper, we can verify that the conjecture holds for \( n \geq 3k \). We omit the proof here.
However, fully resolving the conjecture appears to be a highly challenging problem.

\section*{Acknowledgements}
The authors thank Xizhi Liu for bringing us the references \cite{BCL,BL}, which help us get started on this topic.
Additionally,  the authors thank the anonymous reviewers for their detailed and constructive comments which greatly improved the presentation of this paper.
The first author is supported by the Scientific Research Fund of Hunan Provincial Education Department (No. 25A0080).
The second author is supported by the National Natural Science Foundation of China  (No.12371332 and
No.11971439).

\end{document}